\documentclass[11pt, reqno]{amsart}
\usepackage{geometry}
\usepackage{graphicx}
\usepackage{amssymb}
\usepackage{amsmath, mathrsfs, stmaryrd}
\usepackage{color,ulem}
\newtheorem{theorem}{Theorem}
\newtheorem{definition}[theorem]{Definition}

\newtheorem{lem}[theorem]{Lemma}

\newtheorem{cor}[theorem]{Corollary}

\newtheorem{prop}[theorem]{Proposition}
\newtheorem{remark}[theorem]{Remark}

\numberwithin{theorem}{section}
\numberwithin{equation}{section}

\newcommand{\pl}{\partial}

\newcommand{\na}{ \nabla}

\newcommand{\lt}{\left}
\newcommand{\rt}{\right}

\title{Quasi-local energy with respect to a static spacetime}
\author{Po-Ning Chen, Mu-Tao Wang, Ye-Kai Wang, and Shing-Tung Yau}
\date{}
\thanks{P.-N. Chen is supported by NSF grant DMS-1308164, M.-T. Wang is supported by NSF grant DMS-1405152,  and S.-T. Yau is supported by NSF grants  PHY-0714648 and DMS-1308244.  } 
\begin{document}

\begin{abstract}
This article considers the quasi-local energy in reference to a general static spacetime. We follow the approach developed by the authors in \cite{Wang-Yau1,Wang-Yau2,Chen-Wang-Yau3,Chen-Wang-Yau} and define the quasi-local energy as a difference of surface Hamiltonians, which are derived from the Einstein-Hilbert action.  The new quasi-local energy provides an effective gauge independent measurement of how far a spacetime deviates away from the reference static spacetime on a finitely extended region. 
\end{abstract}
\maketitle
\section{Introduction}
Due to the lack of energy density by Einstein's equivalence principle, the definition of gravitational energy has been a challenging problem. One can at best hope to define energy as a boundary integral instead of a bulk integral. The application of the Hamilton-Jacobi theory \cite{by2, hh} to the Einstein-Hilbert action gives an expression that depends on a reference term. For an isolated system with suitable decay at infinity, it is possible to choose  an asymptotically flat coordinate system to anchor the reference term, and this leads to the celebrated definitions of the ADM energy \cite{Arnowitt-Deser-Misner}, and the positive energy theorems of Schoen-Yau \cite{SY1}, Witten \cite{W}, etc.  However, for a finitely extended system, the choice of a reference had remained subtle and ambiguous until \cite{Wang-Yau1,Wang-Yau2} in which isometric embeddings into the Minkowski spacetime were applied to give a well-defined definition of quasi-local energy. The idea is to utilize the surface Hamiltonian \cite{by2, hh} from the Einstein-Hilbert action to pick up an optimal one among all such isometric embeddings. The resulting definition of energy and conserved quantities have had several remarkable applications \cite{Chen-Wang, Chen-Wang-Yau3, Chen-Wang-Yau2} since then. This approach was subsequently  generalized to define quasi-local energy with respect to de Sitter/Anti-de Sitter reference recently \cite{Chen-Wang-Yau}. In this paper, we further generalize to allow the reference spacetime to be a general static spacetime. Such an energy is not expected to have a straightforward positivity property as the Minkowski reference case. The principal application seems to be to a perturbative configuration. For example, although the black hole uniqueness theorem \cite{Israel,Bunting-Massod} establishes the Schwarzschild solution as the unique asymptotically flat static vacuum spacetime, a black hole in reality will be a perturbation.  The quasi-local energy provides an effective gauge independent measurement of how far such a perturbation deviates away from the exact Schwarzschild solution. 

Throughout this article, a spacetime is a time-oriented Lorentz 4-manifold. We impose the static condition on the reference spacetime. 

\begin{definition}\label{static} A static spacetime is a time-oriented Lorentz 4-manifold (with possibility nonempty smooth boundary) such that there exists a coordinate system 
$(t, x^1, x^2, x^3)$ (static chart) under which the Lorentz metric takes the form
\begin{equation}\label{metric_form}
\check{g}= -V^2(x^1, x^2, x^3) dt^2 +g_{ij} (x^1, x^2, x^3) dx^i dx^j,
\end{equation}
where $V>0$ on the interior and $V=0$ on the boundary.
 
\end{definition}
Each time slice, i.e. the hypersurface defined by $t=c$ for a constant $c$, is a smooth Riemannian 3-manifold with possibly nonempty smooth boundary $\partial M$, such that $V>0$ in the interior of $M$ and $V=0$ on $\partial M$. Denote the covariant derivative of the metric $\check{g}$ by D and that of the metric $g$ by $\bar \na$.

In the following, we recall the null convergence condition:
\begin{definition}\label{null-convergence} 
A spacetime with Lorentz metric $\check{g}$ satisfies the null convergence condition if 
\begin{equation}\label{null_conv} Ric_{\check{g}}(L,L) \ge 0  \end{equation}
for any null vector $L$, where $Ric_{\check{g}}$ is the Ricci curvature of $\check{g}$. 
\end{definition}
Recall that $L$ is  a null vector if $\check{g}(L,L)=0$. By \cite{WWZ}, a static spacetime satisfies the null convergence condition \eqref{null_conv} if and only if
\begin{equation}\label{null}  \bar \Delta V g - \bar \na^2 V + V Ric \ge 0, \end{equation} on each time slice, where $Ric$ is the Ricci curvature of the metric $g$.

In particular, static vacuum spacetimes satisfy the null convergence condition. These spacetimes have been studied extensively. We summarize some basic properties as follows.  
The static metric $\check{g}$ satisfies the vacuum Einstein equation with the cosmological constant $\Lambda$ if 
\begin{equation}\label{vacuum-static}
\left\{
    \begin{array}[]{rlll}
-\Lambda V g - \bar \na^2 V + V Ric &=& 0,\\
        \bar \Delta V + \Lambda V &=& 0. \\
    \end{array}
    \right.
\end{equation}
From  \eqref{vacuum-static}, it follows that (see \cite[Proposition 2.3]{Corvino}, \cite[Lemma 2.1]{Miao-Tam} for example)
\begin{enumerate}
\item{The scalar curvature of $g_{ij}$ is constant,}
\item{0 is a regular value of $V$ and $\{ V=0\}$ is totally geodesic,}
\item{$|\bar \na V|$ is a positive constant on each component of $\{ V=0\}$.}
\end{enumerate}

From here on, we pick a static spacetime $\mathfrak{S}$ as in Definition \ref{static} and refer to it as the reference spacetime.  Let $\mathring{\mathfrak{S}}$ denote the interior and $\partial\mathfrak{S}$ denote the boundary of $\mathfrak{S}$, respectively. In addition, we refer to the hypersurface  $t=c$ as a static slice and the function $V$ as the static potential.

The results in the paper are summarized as follows. The definition of quasi-local energy is given in \S 2.2. For a surface in the reference spacetime, it is proved that the identity isometric embedding not only has energy zero by definition, but also is a critical point of the quasi-local energy (Theorem \ref{thm_own_critical}). The first variation of the quasi-local energy, which characterizes an optimal isometric embedding,  is derived in Theorem \ref{thm_first_variation_graph}. At last, it is shown that the identity isometric embedding of a surface in the static slice is locally energy-minimizing (Theorem \ref{minimize_self_1}).

\section{Quasi-local energy with respect to a static spacetime reference}
In this section, we define a new quasi-local energy allowing the reference spacetime to be a general static spacetime,  following the construction in \cite{Chen-Wang-Yau}.

\subsection{Geometry of surfaces in a static spacetime} Let $\mathfrak{S}$ be a reference spacetime.
Consider a surface $\Sigma$ in  $\mathring{\mathfrak{S}}$ defined by an embedding $X$ of an abstract surface $\Sigma_0$. In the static chart, we denote the components of $X$ by $(\tau,X^1,X^2,X^3)$. Let $\sigma$ be the induced metric on $\Sigma$,  $H_0$ be the mean curvature vector of $\Sigma$, and $J_0$ be the reflection of $H_0$ through the incoming light cone in the normal bundle of $\Sigma$. Denote the covariant derivative with respect to the induced metric $\sigma$ by $\nabla$. 

Given an orthonormal frame $\{ e_3,e_4\}$ of the normal bundle of $\Sigma$ in $\mathring{\mathfrak{S}}$  where $e_3$ is spacelike and $e_4$ is future timelike, we define the connection one-form associated to the frame 
\begin{equation}\label{alpha_3}  \alpha_{e_3} (\cdot)=\langle D_{(\cdot)}  e_3, e_4 \rangle.\end{equation}
We assume the mean curvature vector of $\Sigma$ is spacelike and consider the following connection one-form of $\Sigma$ with respect to the mean curvature vector:
\begin{equation}\label{alpha_h} \alpha_{H_0}(\cdot )=\langle D_{(\cdot)}   \frac{J_0}{|H_0|}, \frac{H_0}{|H_0|}   \rangle.  \end{equation}

Let $\widehat \Sigma$ be the surface in the static slice $t=0$ given by $\widehat X=(0,X^1,X^2,X^3)$ which is assumed to be an embedding. The surfaces $\Sigma$ and $\widehat\Sigma$ are canonically diffeomorphic through the above identification. Let $\hat \sigma$ be the induced metric on $\widehat \Sigma$, and $\widehat H$  and $\hat h_{ab}$ be the mean curvature and second fundamental form of $\widehat \Sigma$ in the static slice, respectively. Denote the covariant derivative with respect to the metric $\hat \sigma$ by $\hat \nabla$. 

Let $C$ be the image of $\Sigma$ under the one-parameter family $\phi_t$. The intersection of $C$ with the static slice $t=0$ is $\widehat \Sigma$. 
Let $\breve e_3$ be the outward unit normal of $\widehat \Sigma$ in the static slice $t=0$. Consider the pushforward of $\breve e_3$ by the one-parameter family  $\phi_t$, which is denoted by $\breve e_3$ again. Let $\breve e_4$ be the future directed unit normal of $\Sigma$ normal to $\breve e_3$ and extend it along $C$ in the same manner. It is easy to see that Lemma 2.1, Proposition 2.2, Proposition 2.3 and Proposition 3.2 of \cite{Chen-Wang-Yau} hold for a general static spacetimes. We state them here for later reference.

\begin{lem}
Along $C$, we have
\begin{align}\label{decompose_4} \breve e_4  = & \sqrt{1+V^2| \nabla \tau|^2} \left(  \frac{\frac {\partial}{\partial t}}{V} +  V \hat \nabla \tau\right)\\
\label{decompose_t}  \frac {\partial}{\partial t} =& V \sqrt{1+V^2| \nabla \tau|^2} \breve e_4 - V^2 \nabla \tau.
\end{align}
\end{lem}

\begin{prop}\label{proposition_mean_curvature_projection}
Along $C$,
\begin{equation}\label{relation_mean_curvature}
\widehat H = -\langle H_0, \breve e_3 \rangle - \frac{V}{ \sqrt{1+V^2| \nabla \tau|^2}} \alpha_{\breve e_3}(\nabla \tau).
\end{equation}
\end{prop}
\begin{prop} \label{connection_reference}
Along $C$, the connection one-form $\alpha_{\breve e_3}$ on $\Sigma$ satisfies
\begin{equation}\label{connection_reference_one}
(\alpha_{\breve e_3})_a =  \sqrt{1+V^2| \nabla \tau|^2} ( V \hat \nabla^b \tau  \hat h_{ab} - \breve e_3 (V) \tau_a )\end{equation}
where $\hat h_{ac}$ on the right hand side is the extension of the second fundamental form of $\widehat \Sigma$ to $C$ by the one-parameter family $\phi_t$.
\end{prop}
\begin{prop} \label{total_mean_mean_gauge}
In terms of the connection one-form in mean curvature gauge $\alpha_{H_0}$, we have
\[
\begin{split}
& \int V \widehat H d \widehat \Sigma = \int  \Big [ \sqrt{(1+V^2| \nabla \tau|^2) |H_0|^2V^2 + div(V^2 \nabla \tau)^2 } + div(V^2 \nabla \tau) \theta  -  \alpha_{H_0} (V^2 \nabla \tau)  \Big ] d \Sigma,
\end{split}
\] where \begin{equation}\label{gauge_angle}
\theta = - \sinh^{-1}  \frac{ div(V^2 \nabla \tau) }{|H_0|V \sqrt{1+V^2| \nabla \tau|^2} }
\end{equation}
and
\begin{equation}
\begin{split} \label{gauge}
- \frac{H_0}{|H_0|} = & \cosh \theta \breve e_3 + \sinh \theta \breve e_4 \\
\frac{J_0}{|H_0|} = & \sinh \theta \breve e_3 + \cosh \theta \breve e_4. \\
\end{split}
\end{equation}
\end{prop}
In particular, 
\begin{equation} \label{gauge_change}   \langle H_0 , \breve e_4 \rangle  = |H_0| \sinh \theta, \quad
  - \langle H_0 , \breve e_3 \rangle  = |H_0| \cosh \theta, 
 \text{ \,\, and  \,\,} 
 \alpha_{H_0} = \alpha_{\breve e_3} - d \theta.
\end{equation} 

\subsection{Definition of quasi-local energy}
Let $\Sigma$ be  a surface  in a  general spacetime $N$ (not necessarily static). We assume the mean curvature vector $H$ of $\Sigma$ is spacelike and the normal bundle of $\Sigma$ is oriented. The data for defining the quasi-local energy consists of the triple $(\sigma,|H|,\alpha_H)$ where $\sigma$ is the induced metric, $|H|$ is the norm of the mean curvature vector, and $\alpha_H$ is the connection one-form of the normal bundle with respect to the mean curvature vector
\[ \alpha_H(\cdot )=\langle \nabla^N_{(\cdot)}   \frac{J}{|H|}, \frac{H}{|H|}   \rangle.  \]
Here $J$ is the reflection of $H$ through the incoming  light cone in the normal bundle. For an isometric embedding $X$ into the interior $\mathring{\mathfrak{S}}$ of a reference spacetime $\mathfrak{S}$ with the static potential $V$, we write $X= (\tau ,X^1,X^2,X^3)$ with respect to a fixed static chart. We define $\widehat{X}, \widehat{\Sigma}, \widehat{H}$ as in the last subsection. The quasi-local energy associated to the pair $(X,\frac{\partial}{\partial t})$ is defined to be
\begin{equation}\label{energy_fix_chart_base}
\begin{split}
  E(\Sigma, X,\frac{\partial}{\partial t})
= & \frac{1}{8 \pi}  \Big \{  \int V \widehat H d \widehat \Sigma -
 \int  \Big [ \sqrt{(1+ V ^2| \nabla \tau|^2) |H|^2   V ^2 + div( V ^2 \nabla \tau)^2 }  \\
& \qquad -   div( V ^2 \nabla \tau)  \sinh^{-1} \frac{ div( V ^2 \nabla \tau) }{ V  |H|\sqrt{1+ V ^2| \nabla \tau|^2} }
  -  V ^2 \alpha_{H} (\nabla \tau)  \Big ] d \Sigma
\Big \}.
\end{split}
\end{equation}

Using Proposition \ref{total_mean_mean_gauge}, we rewrite the quasi-local energy as follows:
\begin{equation}\label{energy_fix_chart_graph} 
\begin{split}
   E(\Sigma, X,\frac{\partial}{\partial t}) 
=& \frac{1}{8 \pi}  \Big \{  
 \int  \Big [ \sqrt{(1+ V ^2| \nabla \tau|^2) |H_0|^2  V ^2 + div( V ^2 \nabla \tau)^2 }  \\
& \qquad -   div( V ^2 \nabla \tau)  \sinh^{-1} \frac{ div( V ^2 \nabla \tau) }{ V  |H_0|\sqrt{1+ V ^2| \nabla \tau|^2} }
  -  V ^2 \alpha_{H_0} (\nabla \tau)  \Big ] d \Sigma\\
& -  \int  \Big [ \sqrt{(1+ V ^2| \nabla \tau|^2) |H|^2   V ^2 + div( V ^2 \nabla \tau)^2 }  \\
& \qquad -   div( V ^2 \nabla \tau)  \sinh^{-1} \frac{ div( V ^2 \nabla \tau) }{ V  |H|\sqrt{1+ V ^2| \nabla \tau|^2} }
  -  V ^2 \alpha_{H} (\nabla \tau)  \Big ] d \Sigma
\Big \}.
\end{split}
\end{equation}

The optimal isometric embeddings is defined as in \cite{Chen-Wang-Yau}.
\begin{definition} Let $\mathfrak{S}$ be a reference spacetime. 
An optimal isometric embedding for the data $(\sigma, |H|, \alpha_H)$ is an isometric embedding $X_0$ of $\sigma$ into $\mathring{\mathfrak{S}}$  that is a critical point of the quasi-local energy $E(\Sigma, X, \frac{\partial}{\partial t})$ among all nearby isometric embeddings $X$ of $\sigma$ into $\mathring{\mathfrak{S}}$.
\end{definition}
We show that for a surface in the interior $\mathring{\mathfrak{S}}$ of the reference static spacetime, the identity embedding is an  optimal isometric embedding.
\begin{theorem} \label{thm_own_critical}
The identity isometric embedding for a surface $\Sigma$ in the interior $\mathring{\mathfrak{S}}$ of $\mathfrak{S}$ is a critical point of its own quasi-local energy. Namely, suppose $\Sigma$ in $\mathring{\mathfrak{S}}$  is defined by an embedding $X_0$. Consider a family of isometric embeddings $X(s)$, $-\epsilon<s<\epsilon$ such that $X(0)=X_0$. Then we have
\[  \frac{d}{ds}|_{s=0} E(\Sigma, X(s), \frac{\partial}{\partial t})= 0. \]
\end{theorem}
\begin{proof}
Denote $\frac{d}{ds}|_{s=0}$ by $\delta$ and set 
\[\frak{H}_1= \int  V  \widehat H d \widehat \Sigma \]
and 
\[\begin{split} \frak{H}_2=& \int  \Big [ \sqrt{(1+ V ^2| \nabla \tau|^2) |H_0|^2   V ^2 + div( V ^2 \nabla \tau)^2 } \\
&-   div( V ^2 \nabla \tau)  \sinh^{-1} \frac{ div( V ^2 \nabla \tau) }{ V  |H_0|\sqrt{1+ V ^2| \nabla \tau|^2} }   -  V ^2 \alpha_{H_0} (\nabla \tau)  \Big ] d \Sigma.\end{split}\]
It suffices to prove that $\delta \frak{H}_1=\delta \frak{H}_2$,
where for the variation of $\frak{H}_2$, it is understood that $H_0$ and $\alpha_{H_0}$ are fixed at their values at the initial surface  $X_0(\Sigma)$ and only $\tau$ and $ V $ are varied.

It is convenient to rewrite   $\frak{H}_1$  and $\frak{H}_2$ in terms of the following two quantities: $A =  V  \sqrt{1+ V ^2|\nabla \tau|^2}$ and $B= div( V ^2 \nabla \tau)$. In terms of $A$ and $B$ 
\[
\begin{split}
\frak{H}_1= &\int \widehat H A \,d \Sigma\\
 \frak{H}_2=&\int \left[ \sqrt{|H_0|^2A^2+B^2} - B \sinh^{-1} \frac{B}{|H_0| A}  - \alpha_{H_0}( V ^2 \nabla \tau) \right] d \Sigma. 
\end{split}
\] 
As a result, we have
\[
\begin{split}
 \delta \frak{H}_2 
=&\int \left [\delta A(\frac{|H_0|^2A}{\sqrt{|H_0|^2A^2+B^2}} + \frac{B^2}{A\sqrt{|H_0|^2A^2+B^2}}) \right]d\Sigma \\
& -\int \left[ (\delta B)  \sinh^{-1} \frac{B}{|H_0|A}+\alpha_{H_0}(\delta(  V ^2 \nabla \tau)) \right] d \Sigma\\
=&\,\mbox{I}-\mbox{II}
\end{split}
\]

By \eqref{gauge_change} and $\sinh \theta = -\frac{B}{|H_0|A}$, integrating by parts gives
\[\mbox{II} =\int \left[ \delta( V ^2 \nabla\tau)\cdot \nabla\theta+\alpha_{H_0}(\delta(  V ^2 \nabla \tau))\right] d\Sigma=\int \alpha_{\breve e_3} (\delta(  V ^2 \nabla \tau)) d\Sigma.
\]

On the other hand, we simplify the integrand of $\mbox{I}$ using \eqref{gauge_angle}, \[\frac{|H_0|^2A}{\sqrt{|H_0|^2A^2+B^2}} + \frac{B^2}{A\sqrt{|H_0|^2A^2+B^2}}=\frac{\sqrt{|H_0|^2A^2+B^2}}{A}=-\langle H_0, \breve{e}_3\rangle.\]

Therefore, by  \eqref{relation_mean_curvature}, $\mbox{I}$ is equal to 
\[
\begin{split}
& \int   (-\langle H_0, \breve e_3\rangle)\delta A d\Sigma \\
=& \int   [\widehat H + \frac{ V  \, \alpha_{\breve e_3} (\nabla \tau)  }{ \sqrt{1+ V ^2| \nabla \tau|^2}}]\delta A d\Sigma\\
= & \int \widehat H\delta Ad\Sigma + \int   \left[\frac{(\delta  V )  V ^3|\nabla \tau|^2 +  V ^4 \nabla \tau \nabla \delta \tau}{1+ V ^2|\nabla \tau|^2}( \alpha_{\breve{e}_3}(\nabla\tau)) +( \delta  V )   V   \alpha_{\breve{e}_3}(\nabla\tau) \right] d\Sigma.
\end{split}
\]
and 
\begin{equation} \label{del_H2}
\begin{split}
 \delta\frak{H}_2&=  \int \widehat H\delta Ad\Sigma  + \int \left[  \frac{(\delta  V )  V ^3|\nabla \tau|^2 +  V ^4 \nabla \tau \nabla \delta \tau}{1+ V ^2|\nabla \tau|^2}( \alpha_{\breve{e}_3}(\nabla\tau)) 
- \alpha_{\breve{e}_3}( V \delta  V   \nabla\tau +  V ^2 \nabla \delta \tau)\right] d\Sigma\\
&=  \int \widehat H\delta Ad\Sigma  - \int  (\alpha_{\breve{e}_3})_a (\sigma^{ac} -\frac{ V ^2 \nabla^a \tau \nabla^c \tau}{1+ V ^2|\nabla \tau|^2})( V  \delta  V   \tau_c+  V ^2 \delta \tau_c) d\Sigma \\
&=  \int \widehat H\delta Ad\Sigma  +\int - (\alpha_{\breve{e}_3})_a \hat \sigma^{ac}(  V \delta  V   \tau_c+  V ^2 \delta \tau_c) d\Sigma.
\end{split}
\end{equation}
Applying  Proposition \ref{connection_reference}, the second integral in the last line can be rewritten as

\[
\begin{split}
& \int   \sqrt{1+ V ^2| \nabla \tau|^2}( \breve e_3 ( V ) \tau_a -  V  \hat  \nabla^b \tau \hat h_{ab}) \hat \sigma^{ac}( V  \delta  V   \tau_c+  V ^2 \delta \tau_c) d \Sigma  \\
= & \int [\breve e_3( V )\hat \sigma^{ab}-  V  \hat h^{ab} ]  (  V  \delta  V   \tau_a \tau_b +  V ^2 \tau_a \delta \tau_b)  d \widehat \Sigma \\
= &  \frac{1}{2}\int [\breve e_3( V )\hat \sigma^{ab}-  V  \hat h^{ab} ]  (\delta \hat  \sigma)_{ab}  d \widehat \Sigma.
\end{split}
\]

On the other hand,  as $ V  d\widehat\Sigma=A d\Sigma$ and $\delta d\Sigma=0$,
\begin{equation}\label{del_H1}
\begin{split}
\delta\frak{H}_1=   \int \widehat H\delta Ad\Sigma  +  \int  V  \delta \widehat H   d \widehat \Sigma . \\
\end{split}
\end{equation}

To prove $\delta \frak{H}_1=\delta\frak{H}_2$, by  \eqref{del_H2} and \eqref{del_H1}, it suffices to show
\begin{equation}\label{eq_hat}\int  V  \left[ \delta \widehat H +\frac{1}{2}  \hat h^{ab} (\delta \hat  \sigma)_{ab}\right]d \widehat \Sigma=\frac{1}{2}\int \left[\breve e_3( V )\hat \sigma^{ab}  (\delta \hat  \sigma)_{ab}\right]  d \widehat \Sigma.\end{equation}

We decompose $\delta \widehat X$ into tangential and normal parts to $\widehat \Sigma$. Let
\[ \delta \widehat X = P^a \frac{\partial \widehat X}{\partial v^a} + \beta \nu.  \]
For the first and second variations of the induced metric, we have
\begin{equation}
\label{first_variation_metric} (\delta \hat \sigma)_{ab}   =2 \beta \hat h_{ab} + \hat \nabla_a( P^c \hat \sigma_{cb}) + \hat \nabla_b (P^c \hat \sigma_{ca}) 
\end{equation}
and
\begin{equation}\label{second_variation_metric_1} 
\begin{split}
 \delta \widehat H =& - \widehat  \Delta \beta  - Ric(e_3,e_3) \beta -\beta \hat \sigma^{ab} \hat \sigma^{dc} \hat h_{ac} \hat h_{bd} + P^a \hat \nabla_a \widehat H \\
=&  - \widehat  \Delta \beta  - Ric(e_3,e_3) \beta -\beta \hat \sigma^{ab} \hat \sigma^{dc} \hat h_{ac} \hat h_{bd} + P^a \hat \nabla^b \hat h_{ab}-P^c Ric(\frac{\pl \widehat X }{\pl v^c},e_3),
\end{split}
\end{equation} where the Codazzi equation is used. 

We derive from \eqref{first_variation_metric} and \eqref{second_variation_metric_1}
\begin{equation}
\label{second_variation_metric}  \delta \widehat H +\frac{1}{2} \hat h^{ab}  (\delta \hat \sigma)_{ab} =- \widehat \Delta \beta - Ric(e_3,e_3) \beta +   \hat \nabla^b( P^c \hat h_{cb})-P^c Ric(\frac{\pl \widehat X }{\pl v^c},e_3).\end{equation}
\eqref{eq_hat} is thus equivalent to 
\[
\begin{split}
\int   V   [- \widehat \Delta \beta - Ric(e_3,e_3) \beta +  \hat \nabla^b( P^c \hat h_{cb})-P^c Ric(\frac{\pl \widehat X }{\pl v^c},e_3)] d \widehat \Sigma=\int \breve e_3( V ) [ \beta \widehat H + \hat \nabla^b( P^c \hat \sigma_{cb})]  d \widehat \Sigma . 
\end{split}
\]
The above equality follows from the following two identities:
\begin{align}
\label{equality_b}\int \breve e_3( V )  \beta \widehat H  d \widehat \Sigma = & \int   V   [- \widehat \Delta \beta -Ric(e_3,e_3) \beta]  d \widehat \Sigma \\
\label{equality_a}\int \breve e_3( V ) \hat \nabla^b( P^c \hat \sigma_{cb}) d \widehat \Sigma = & \int   V   [  \hat \nabla^b (P^c \hat h_{cb}) -Ric(e_3,c)]d \widehat \Sigma,
   \end{align} which can be derived by integrating by parts and the static equation.
\end{proof}
We define the quasi-local energy density $\rho$ with respect to the isometric embedding $X$.
\begin{definition} The quasi-local energy density with respect to the isometric embedding $X$ is defined to be 
\begin{equation} \label{rho} \begin{split}\rho &= \frac{\sqrt{|H_0|^2 +\frac{(div V^2 \nabla \tau)^2}{V^2+V^4 |\nabla \tau|^2}} - \sqrt{|H|^2 +\frac{(div V^2 \nabla \tau)^2}{V^2+V^4 |\nabla \tau|^2}} }{ V\sqrt{1+ V^2|\nabla \tau|^2}}. \end{split}\end{equation}
\end{definition}
An immediate consequence of Theorem \ref{thm_own_critical} is the following formula for the first variation of the quasi-local energy:
\begin{theorem} \label{thm_first_variation_graph}
Let $\Sigma$ be a surface in a physical spacetime with the data  $(\sigma,|H|, \alpha_H)$. Let $X_0$ be an isometric embedding of $\sigma$ into the interior  $\mathring{\mathfrak{S}}$ of the reference spacetime and let $(|H_0|, \alpha_{H_0})$
 be the corresponding data on $X_0(\Sigma)$. Consider a family of isometric embeddings $X(s)$, $-\epsilon<s<\epsilon$ such that $X(0)=X_0$. Then we have
\begin{equation}\label{first_variation_graph}
\begin{split}
   & \frac{d}{ds}|_{s=0} E(\Sigma, X(s),\frac{\partial}{\partial t}) \\
=& \frac{1}{8 \pi} \int  \delta V \left [ \rho   V  (1+ 2  V ^2|\nabla \tau|^2)  -2   V  \nabla \tau \nabla \sinh^{-1} \frac{\rho div ( V ^2 \nabla \tau)}{|H_0||H|}  + (\alpha_{H} - \alpha_{H_0})(2  V  \nabla \tau)  \right ]   d \Sigma \\
 & +\frac{1}{8 \pi} \int (\delta \tau)   div\left [   V ^2 \nabla \sinh^{-1} \frac{\rho div ( V ^2 \nabla \tau)}{|H_0||H|}  - \rho  V ^4 \nabla \tau + V ^2(\alpha_{H_0} - \alpha_H) \right ]   d \Sigma,
\end{split}  
\end{equation} where $\delta\tau=\frac{d}{ds}|_{s=0} \tau(s)$, $\delta X^i=\frac{d}{ds}|_{s=0} X^i(s)$ and $\delta V=\delta X^i \bar\nabla_i  V$.
\end{theorem}
\begin{proof}
The proof is identical to the proof of Theorem 5.4 of \cite{Chen-Wang-Yau} where Theorem 5.3 of \cite{Chen-Wang-Yau} is replaced by Theorem \ref{thm_own_critical} above.
\end{proof}

\section{A Reilly-type formula for static manifolds}
In this section, we generalize Lemma 6.1 of \cite{Chen-Wang-Yau} for de Sitter and anti-de Sitter spacetimes to general static spacetimes. The proof  of \cite[Lemma 6.1]{Chen-Wang-Yau} relies on a Reilly-type formula for functions on space forms in \cite{Qiu-Xia}. We first prove a Reilly-type formula for a pair $(V,Y)$ of a positive function $V$ and a one-form $Y$ on a Riemannian manifold $(M,g) $ following the recent work of \cite{LX}. Then we apply the Reilly-type formula for the pair to the case where $V$ is the static potential of the reference spacetime. 

Let $(M,g)$ be a Riemannian n-manifold and $\bar \na$ and $\bar \Delta$ be the covariant derivatives and the Laplace operator with respect to $g$. Let $\Omega$ be a bounded domain with smooth boundary $\pl\Omega$ in $M$ . Let $\mathbb{II}$ and $H$ be the second fundamental form and mean curvature of $\pl\Omega$  and $\na$ be the covariant derivative on $\pl\Omega$.

\begin{prop}\label{Reilly-form} Let $V$ be a positive function on $\Omega$ and $Y$ be a one-form on $\Omega$.
Let $Y^T$ be the tangential component of $Y$ to $\pl\Omega$. We have the following integral identity
\begin{align}\label{integral identity}
\begin{split}
&\int_{\pl\Omega} \lt [ - \frac{1}{V} \mathbb{II}(Y^T, Y^T) + \frac{1}{V^2} \frac{\pl V}{\pl \nu} |Y^T|^2 - \frac{1}{V} H \langle Y,\nu \rangle^2 - \frac{2}{V} \na_a (Y^T)^a \langle Y,\nu \rangle \rt ]  dA \\
=& \int_\Omega  \Bigg [\frac{1}{V^2} \lt( \bar \Delta V g - \bar \na^2 V + V Ric \rt)(Y,Y) + \frac{1}{4V} | \bar \na_iY_j + \bar \na_j Y_i|^2 -\frac{1}{V} (\bar \na_i Y^i)^2 \\
   & \quad - \frac{V^3}{4}\lt | \bar \na_i \lt( \frac{Y_j}{V^2}\rt) - \bar \na_j \lt( \frac{Y_i}{V^2}\rt) \rt |^2  \Bigg] d\Omega.
\end{split}
\end{align}
where $\nu$ is the outward unit normal of $\pl\Omega$.
\end{prop}
\begin{proof}
We write $V^i$ for $\bar\na^i V$ in the proof and apply the Bochner formula to $ \frac{1}{V} |Y|^2$:
\begin{align*}
&\frac{1}{2} \bar \Delta \lt( \frac{1}{V} |Y|^2 \rt) \\
=& -\frac{\bar \Delta V	}{2V^2 } |Y|^2 + \frac{|\bar \na V|^2}{V^3} |Y|^2 + \underbrace{ \lt( -\frac{3}{2} \frac{1}{V^2} V^i \bar \na_i |Y|^2 + \frac{1}{V^2} V^i \bar \na_i Y_j Y^j \rt)}_{\mbox{I}} + \frac{1}{2V}  \bar \Delta |Y|^2.
\end{align*}
The last term can be treated as in the classical Bochner formula for one-form:
\begin{align*}
&\frac{1}{V} \bar \na^i \bar \na_j Y_i Y^j + \frac{1}{V} \bar \na^i(\bar \na_iY_j- \bar \na_j Y_i) Y^j + \frac{1}{V} \bar \na_i Y_j D^i Y^j \\
=& \frac{1}{V} \bar \na_j \bar \na^i Y_i Y^j + \frac{1}{V} R_{jk} Y^k Y^j +  \frac{1}{V} \bar \na^i(\bar \na_iY_j- \bar \na_j Y_i) Y^j + \frac{1}{V} \bar \na_i Y_j \bar \na^i Y^j\\
=& \bar \na_j  \lt (\frac{1}{V} \bar \na^i Y_i Y^j  \rt ) -  \frac{1}{V}( \bar \na^i Y_i) ^2 + \frac{1}{V^2} \bar \na^i Y_i Y^j V^j
+ \frac{1}{V} R_{jk} Y^k Y^j  \\
 & +   \bar \na^i \lt ( \frac{1}{V} (\bar \na_iY_j- \bar \na_j Y_i) Y^j \rt ) + \frac{1}{V} \bar \na_i Y_j \bar \na^j Y^i + \frac{1}{V^2} (\bar \na_iY_j- \bar \na_j Y_i)V^i Y^j .
\end{align*}
For term $\mbox{I}$, we have
\begin{align*}
\mbox{I}&= -\frac{3}{2} \bar \na_i \lt( \frac{1}{V^2} V^i |Y|^2 \rt) + \frac{3}{2} \frac{\bar \Delta V}{V^2} |Y|^2 - 3 \frac{|\bar \na V|^2}{V^3} |Y|^2 \\
&\quad + \bar \na_j \lt( \frac{1}{V^2} V^i Y_i Y^j \rt) + \frac{2}{V^3} \langle \bar \na V,Y \rangle^2 - \frac{\bar \na_i\bar \na_j V}{V^2} Y^i Y^j\\
&\quad - \frac{1}{V^2} \langle \bar \na V, Y \rangle \bar \na_j Y^j + \frac{1}{V^2} V^i (\bar \na_i Y_j - \bar \na_j Y_i) Y^j.
 \end{align*}
Collecting terms, we obtain
\begin{align*}
\frac{1}{2} \bar \Delta \lt( \frac{1}{V} |Y|^2 \rt) =& \frac{1}{V^2} \lt( \bar \Delta V g - \bar \na^2 V + V Ric \rt)(Y,Y) -\frac{1}{V} (\bar \na_i Y^i)^2  \\
& \underbrace{ - 2 \frac{|\bar \na V|^2|Y|^2}{V^3} + 2 \frac{\langle \bar \na V,Y \rangle^2}{V^3} + \frac{2}{V^2} (\bar \na_i Y_j - \bar \na_j Y_i) V^i Y^j + \frac{1}{V} \bar \na_j Y_i \bar \na^i Y^j }_{\mbox{II}} \\
& + \bar \na^i \lt( \frac{1}{V} (\bar \na_i Y_j - \bar \na_j Y_i) Y^j \rt) - \frac{3}{2} \bar \na_i \lt( \frac{1}{V^2} V^i |Y|^2 \rt) \\  &
+ \bar \na_j \lt( \frac{1}{V^2} \langle \bar \na V,Y \rangle Y^j \rt) + \bar \na_j \lt( \frac{1}{V} \bar \na_i Y^i Y^j \rt) .
\end{align*}
Making the substitution
\begin{align}\label{substitution}
\bar \na_i Y_j - \bar \na_j Y_i = \frac{2}{V} (V_i Y_j - V_j Y_i) +  V^2 \lt[ \bar \na_i \lt( \frac{Y_j}{V^2}\rt) - \bar \na_j \lt( \frac{Y_i}{V^2}\rt) \rt] 
\end{align}
and
\[ \bar\na_i Y_j = \frac{1}{2} \lt( \bar\na_i Y_j + \bar\na_j Y_i \rt) + \frac{1}{2} \lt( \bar\na_i Y_j - \bar\na_j Y_i \rt), \]
we get
\begin{align*}
\mbox{II} &= 2 \frac{|\bar \na V|^2|Y|^2}{V^3} - 2 \frac{\langle \bar \na V,Y \rangle^2}{V^3}  + 2 \lt[ \bar\na_i \lt( \frac{Y_j}{V^2}\rt) - \bar\na_j \lt( \frac{Y_i}{V^2} \rt)\rt] V^i Y^j \\
&\quad + \frac{1}{4V} |\bar\na_i Y_j + \bar\na_j Y_i|^2 - \frac{1}{4V} |\bar\na_i Y_j - \bar\na_j Y_i|^2 \\
&= \frac{1}{4V} |\bar\na_i Y_j + \bar\na_j Y_i|^2 - \frac{V^3}{4} \lt| \bar\na_i \lt( \frac{Y_j}{V^2}\rt) - \bar\na_j \lt( \frac{Y_i}{V^2}\rt) \rt|^2. 
\end{align*}
Here (\ref{substitution}) is used again in the last equality. In summary, we obtain
\begin{align*}
\frac{1}{2} \bar\Delta \lt( \frac{1}{V} |Y|^2 \rt) =&  \frac{1}{V^2} \lt( \bar \Delta V g - \bar \na^2 V + V Ric \rt)(Y,Y) -\frac{1}{V} (\bar \na_i Y^i)^2 \\
&+\frac{1}{4V} |\bar\na_i Y_j + \bar\na_j Y_i|^2 - \frac{V^3}{4} \lt| \bar\na_i \lt( \frac{Y_j}{V^2}\rt) - \bar\na_j \lt( \frac{Y_i}{V^2}\rt) \rt|^2 \\
& + \bar \na^i \lt( \frac{1}{V} (\bar \na_i Y_j - \bar \na_j Y_i) Y^j \rt) - \frac{3}{2} \bar \na_i \lt( \frac{1}{V^2} V^i |Y|^2 \rt) \\  &
+ \bar \na_j \lt( \frac{1}{V^2} \langle \bar \na V,Y \rangle Y^j \rt) + \bar \na_j \lt( \frac{1}{V} \bar \na_i Y^i Y^j \rt) .
\end{align*}
Integrating by parts, we get
\begin{align*}
&\int_{\pl\Omega} \Big [ \frac{1}{2} \frac{\pl}{\pl \nu} \lt( \frac{1}{V} |Y|^2 \rt) - \frac{1}{V} (\bar \na_i Y_j - \bar \na_j Y_i) \nu^i Y^j + \frac{3}{2V^2} \frac{\pl V}{\pl\nu} |Y|^2 \\
& \quad - \frac{1}{V^2} \langle \bar \na V, Y \rangle \langle Y,\nu \rangle  - \frac{1}{V} \bar \na_i Y^i \langle Y,\nu \rangle \Big] d A \\
=& \int_\Omega \frac{1}{V^2} \lt( \bar \Delta V g - \bar \na^2 V + V Ric \rt)(Y,Y) + \frac{1}{4V} |\bar \na_i Y_j+\bar \na_j Y_i|^2  -\frac{1}{V} (\bar \na_i Y^i)^2  \\
   & \quad - \frac{V^3}{4}\lt | \bar \na_i \lt( \frac{Y_j}{V^2}\rt) - \bar \na_j \lt( \frac{Y_i}{V^2}\rt) \rt |^2  d \Omega.
\end{align*}
Let's turn to the boundary integral.  We compute
\begin{align*}
\frac{1}{2} \frac{\pl}{\pl \nu} \lt( \frac{1}{V} |Y|^2 \rt)=-\frac{1}{2} \frac{1}{V^2} \frac{\pl V}{\pl\nu} |Y|^2 + \frac{1}{V} \langle \bar \na_Y Y,\nu \rangle + \frac{1}{V} \lt( \langle \bar \na_\nu Y,Y \rangle - \langle \bar \na_Y Y,\nu \rangle\rt), 
\end{align*}
and the boundary integral becomes
\[ \int_{\pl\Omega}\left[ \frac{1}{V} \langle \bar \na_Y Y,\nu \rangle + \frac{1}{V^2} \frac{\pl V}{\pl\nu} |Y|^2 - \frac{1}{V^2} \langle \bar \na V, Y \rangle \langle Y,\nu \rangle  - \frac{1}{V} \bar \na_i Y^i \langle Y,\nu \rangle \right] dA.\]
Decomposing $Y$ into tangential part and normal part to $\pl \Omega$ and using the identity 
\[\bar \na_i Y^i = \na_a (Y^T)^a + \langle \bar \na_\nu Y, \nu \rangle + H \langle Y,\nu \rangle\]
along $\pl \Omega$, we have
\begin{align*}
&\frac{1}{V} \langle \bar \na_Y Y, \nu \rangle  - \frac{1}{V} \bar \na_i Y^i \langle Y,\nu \rangle \\
=& \frac{1}{V} \langle \bar \na_{Y^T + \langle Y,\nu \rangle \nu} Y^T + \langle Y,\nu \rangle  \nu, \nu \rangle  - \frac{1}{V} \bar \na_i Y^i \langle Y,\nu \rangle \\
= &- \frac{1}{V} \mathbb{II}(Y^T, Y^T) + \frac{1}{V} Y^T(\langle Y,\nu \rangle) + \frac{1}{V} \langle \bar \na_\nu Y, \nu \rangle \langle Y,\nu \rangle - \frac{1}{V} \bar \na_i Y^i \langle Y,\nu \rangle \\
=& - \frac{1}{V} \mathbb{II}(Y^T, Y^T) + \frac{1}{V} (Y^T)^a \na_a \langle Y,\nu\rangle - \frac{1}{V} \na_a (Y^T)^a \langle Y,\nu \rangle - \frac{1}{V} H \langle Y,\nu \rangle^2.
\end{align*}
Integrating  by parts the term $\frac{1}{V} (Y^T)^a \na_a \langle Y,\nu\rangle $, we get 

\begin{align*}
  & \int_{\pl\Omega} \left[ \frac{1}{V} \langle \bar \na_Y Y,\nu \rangle  - \frac{1}{V} \bar \na_i Y^i \langle Y,\nu \rangle  \right ] d A \\
= & \int_{\pl\Omega} \left[ - \frac{1}{V} \mathbb{II}(Y^T, Y^T) + \frac{1}{V^2} \langle \na V, Y^T \rangle \langle Y, \nu \rangle - \frac{2}{V} \na_a (Y^T)^a \langle Y,\nu \rangle  - \frac{1}{V} H \langle Y,\nu \rangle^2 \right] d A \\
= & \int_{\pl\Omega}\Big[ - \frac{1}{V} \mathbb{II}(Y^T, Y^T) + \frac{1}{V^2} \langle  \bar \na V, Y \rangle\langle Y,\nu \rangle - \frac{1}{V^2} \frac{\pl V}{\pl\nu} \langle Y,\nu \rangle^2  -\frac{1}{V} H \langle Y,\nu \rangle^2\\
& \quad - \frac{2}{V} \na_a (Y^T)^a \langle Y,\nu \rangle \Big] dA.
\end{align*}
This finishes the proof of the Proposition.
\end{proof}
In particular, for any smooth function $f$ on $\Omega$, we apply Proposition \ref{Reilly-form} to the one-form $Y= V \bar \na f - f \bar \na V$ and derive the following:
\begin{cor} \label{Reilly-function}
Let $f$ be a function on $\Omega$ and define the one-form  $Y =  V \bar \na  f - f   \bar \na  V$. We have
\begin{equation} \label{Reilly-final}
\begin{split}
    & \int_\Omega \left[  \frac{1}{V^2} \left(\bar \Delta V g -  \bar \na ^2 V + V Ric \right)(Y,Y) + \frac{1}{V} |V \bar \na ^2 f  - f \bar \na^2 V|^2   - ( V \bar \Delta f - f \bar\Delta V )^2 \right] d \Omega \\
= & \int_{\pl\Omega} \left [- \frac{1}{V} \mathbb{II}(Y^T, Y^T)  - \frac{2}{V} \nabla_a (Y^T)^a \langle Y,\nu \rangle - \frac{1}{V} H \langle Y,\nu \rangle^2 
+ \frac{1}{V^2} \frac{\partial V}{\partial\nu} |Y^T|^2  \right]
d A.
\end{split}
\end{equation}
\end{cor}
\begin{proof}
We observe that for $Y =  V \bar \na  f - f   \bar \na  V$, 
\[  
\begin{split}
\bar \na_i \lt( \frac{Y_j}{V^2}\rt) - \bar \na_j \lt( \frac{Y_i}{V^2}\rt) =&0,\\
\bar \na_i Y_j + \bar \na_j Y_i  =&2(V\bar \na_i \bar \na_j f - f \bar \na_i \bar \na_j V).
\end{split}
\] 
The corollary follows immediately from Proposition  \ref{Reilly-form}.
\end{proof}

We apply Corollary \ref{Reilly-function} to obtain the following positivity result.
\begin{theorem} \label{Reilly-Positivity}
Suppose $(M, g)$ is a Riemannian manifold and $V$ is a smooth function such that the triple $(M,g,V)$ satisfies the null convergence condition \eqref{null}. Let $\Sigma$ be a closed connected mean convex hypersurface in $M$. Suppose $\Sigma$ bounds a domain $\Omega$ in $M$ such that $\partial  \Omega = \Sigma  \cup N$ where $N$ is contained in $ \partial M.$ For any $\tau \in C^\infty(\Sigma)$, we have
\begin{align}
\int_{\Sigma}  \frac{[ \na_a (V^2 \na^a \tau)]^2}{VH} - V^3 \mathbb{II}(\na \tau, \na\tau) + V^2 \frac{\pl V}{\pl \nu} |\na\tau|^2  d\Sigma\ge 0. \label{second variation>0}
\end{align}
\end{theorem}
\begin{proof}
By Lemma 2.5 of \cite{LX}, the Dirichlet boundary value problem 
\begin{equation}\label{Dirichlet}
\left\{
    \begin{array}[]{rlll}
V \bar \Delta f - f \bar \Delta V &=&0&\mathrm{ in }\ \Omega,\\
        f&=& V\tau &\mathrm{ on }\  \Sigma,\\
        f&=& 0 &\mathrm{ on }\ N,\\
    \end{array}
    \right.
\end{equation}
admits a unique solution $f$.

Consider the one-form $Y=V \bar \na f - f\bar\na V$. By a direct computation, \eqref{second variation>0} is equivalent to
\begin{align*}
\int_{\Sigma} \frac{[\na_a (Y^T)^a]^2}{VH} - \frac{1}{V} \mathbb{II}(Y^T, Y^T) + \frac{1}{V^2} \frac{\pl V}{\pl \nu} |Y^T|^2 d\Sigma \ge 0,
\end{align*}
where $Y^T =V \na f  - f \na V$. 

On the other hand, $Y=0$ on $N$ and \eqref{Reilly-final} is  the same as
\begin{align*}
&\int_\Sigma \frac{[\nabla_a (Y^T)^a]^2}{VH} - \frac{1}{V} \mathbb{II}(Y^T, Y^T) + \frac{1}{V^2} \frac{\partial V}{\partial \nu} |Y^T|^2 d\Sigma \\
= & \int_\Sigma \frac{1}{V} \left( \sqrt{H} \langle Y,\nu \rangle + \frac{\nabla_a (Y^T)^a}{\sqrt{H}} \right)^2 d\Sigma  \\
&\quad + \int_\Omega \frac{1}{V^2} \left( \bar \Delta V g - \bar \na^2 V + V Ric \right)(Y,Y) + \frac{1}{V} |V\bar \na^2 f  - f \bar \na^2 V|^2 d\Omega.
\end{align*}
The assertion follows from \eqref{null}.
\end{proof}


\section{Positivity of the second variation}
In this section, we prove that a convex surface in the static slice of the reference spacetime is a local minimum of its own quasi-local energy. For this result, we assume that the isometric embedding into the static slice is  infinitesimally rigid and the reference spacetime satisfies the null convergence condition. 

\begin{definition}
An isometric embedding into the static slice is infinitesimally rigid if the kernel of the linearized isometric embedding equation consists of the restriction of the Killing vector fields of the static slice to the surface.
\end{definition}

\begin{theorem} \label{minimize_self_1}
Suppose the  reference spacetime $\mathfrak{S}$ satisfies the null convergence condition \eqref{null_conv}. Let
$X(s)=(\tau(s),X^i(s)), s\in (-\epsilon, \epsilon)$ be a family of isometric embeddings of the same metric $\sigma$ into the interior $\mathring{\mathfrak{S}}$  such that the image of $X(0)$ is a convex surface $\Sigma_0$ in the static slice, then
\[
\begin{split}
\frac{d^2}{ds^2}|_{s=0} E(\Sigma_0,X(s), \frac{\partial}{\partial t}) \ge & 0\\
\end{split}
\]
if the isometric embedding of $\Sigma_0$ into the static slice is infinitesimally rigid and  $\Sigma_0$ bounds a domain in the static slice.
\end{theorem}
\begin{proof}
Let $H_0(X(s))$ and $\alpha_{H_0}(X(s))$ be the mean curvature vector and the connection one-form in mean curvature gauge of the image of $X(s)$. For simplicity, set $\delta |H_0| = \frac{d}{ds}|_{s=0} |H_0(X(s))|$ and
$\delta \alpha_{H_0} =\frac{d}{ds}|_{s=0} \alpha_{H_0}(X(s))$. Let $\widehat X(s)=(0,X^i(s))$ be the projection of $X(s)(\Sigma)$ onto the static slice. $\widehat X(s)$ is an isometric embedding of the metric 
\[  \hat \sigma(s)_{ab}= \sigma_{ab} +  V ^2(s) \tau_a(s) \tau_b(s) \]
into the static slice and $ \delta \hat \sigma = \frac{d}{ds}|_{s=0} \hat \sigma(s)=0$, as $\tau(0)=0$.

From the infinitesimal rigidity of the isometric embeddings into the static slice,  there is a family of isometries $\hat A(s)$ of the static slice with $\hat A(0)=Id$  such that 
\[
   \delta \hat A =  \delta  \widehat X
\]
along the surface $\Sigma_0$. Here we set $ \delta \hat A = \frac{d}{ds}|_{s=0} \hat A (s)$  and $\delta \widehat X = \frac{d}{ds}|_{s=0} \widehat X (s) $. Moreover, there is a family $A(s)$ of isometries of the reference spacetime whose restriction to the static slice is  the family $\hat A (s)$. Consider the following family of isometric embeddings of $\sigma$ into the reference spacetime:
\[  \breve X(s) = A^{-1}(s) X(s). \]
 Suppose $\breve X (s) = (\breve \tau(s) , \breve X^i(s))$ in the fixed static coordinate, we have \begin{equation} \label{match_first} \frac{d}{ds}|_{s=0} \breve X^i (s)= 0.\end{equation}

We claim that 
\begin{equation}\label{equality_second_variation}
\frac{d^2}{ds^2}|_{s=0} E(\Sigma_0,X(s), \frac{\partial}{\partial t}) = \frac{d^2}{ds^2}|_{s=0} E(\Sigma_0,\breve X(s), \frac{\partial}{\partial t}).
\end{equation}

Let $H_0(\breve X(s))$ and $\alpha_{H_0}(\breve X(s))$ be the the mean curvature vector and the connection one-form in mean curvature gauge of the images of  $\breve X(s)$. 
\begin{equation}\label{global_isometry_invariant}
\begin{split}
|H_0(X(s))| = &  |H_0(\breve X(s))| \\
\alpha_{H_0} (X(s))=&\alpha_{H_0} (\breve X(s))
\end{split}
\end{equation}
since both are invariant under isometries of the reference spacetime.  By \eqref{match_first}, is easy to see that 
\begin{equation} \label{variation_mean_vanish}  \frac{d}{ds}|_{s=0} | \breve H_0(s)| = 0.   \end{equation}

Moreover, while $\breve \tau(s)$ is different from $\tau(s)$, we have 
\begin{equation}\label{time_function_same} \frac{d}{ds}|_{s=0} \breve \tau(s) = \frac{d}{ds}|_{s=0} \tau(s)  =f \end{equation}
since $\tau(0) =0$, $A(0)= Id$ and the static slice is invariant under the action of $A(s)$.  

We apply Theorem  \ref{thm_first_variation_graph}  to each of $X(s)(\Sigma)$ and $\breve X(s)(\Sigma)$ and use \eqref{global_isometry_invariant}, \eqref{variation_mean_vanish} and \eqref{time_function_same} to differentiate
\eqref{first_variation_graph} one more time.  Only the derivative of the term $ \frac{1}{8 \pi} \int_{\Sigma} (\delta \tau)   div (  V ^2 \alpha_{H_0} )  d \Sigma $ survives after the evaluation at $s=0$.  We thus  conclude that both sides of \eqref{equality_second_variation} are the same as
\[ 
   - \frac{1}{8 \pi} \int  ( \delta \alpha_{H_0} )( V  ^2 \nabla f) d\Sigma_0 . 
\]
Differentiating  \eqref{connection_reference_one}, \eqref{gauge_angle} and \eqref{gauge_change} with respect to $s$, we conclude that
\[ (\delta \alpha_{H_0})_a = \nabla_a \left( \frac{ div (  V ^2 \nabla f)}{ V  |H_0|}\right)+V  h_{ab} \nabla^b f  - f_a e_3( V ).\]
As a result,
\[
\begin{split}
    & - \int  ( \delta \alpha_{H_0} )( V  ^2 \nabla f) d\Sigma_0 \\
= & - \int    V ^2 f^a [\nabla_a \left( \frac{ div (  V ^2 \nabla f)}{ V  |H_0|}\right) +  V  h_{ab} \nabla^b f  - f_a e_3( V ) ]  d\Sigma_0\\
= & \int \left \{ \frac{[div( V ^2 \nabla f)]^2}{|H_0|  V } -  V ^3 h^{ab} f_af_b +  V ^2 |\nabla f|^2 e_3 ( V ) \right \} d\Sigma_0.
\end{split}
\]
The theorem follows from Theorem \ref{Reilly-Positivity}.
\end{proof}
In \cite[Theorem 4]{Li-Wang}, it is proved that in a spherically symmetric $3$-manifold with metric
\[ g= \frac{1}{f^2(r)}dr^2+r^2dS^2, \]
the sphere of symmetry $r=c$ is not infinitesimally rigid unless $g$ is a space form. From the symmetry, it is easy to see that the sphere of symmetry  is of constant mean curvature (CMC). In the following theorem, we prove that the conclusion for Theorem \ref{minimize_self_1} still holds for the sphere of symmetry if it is a stable CMC surface.
\begin{definition}
A CMC surface $\Sigma$ is stable if 
\begin{equation}
\int \lt [|\nabla f|^2 - (|h|^2 + Ric(\nu,\nu))f^2\rt ] d\Sigma \ge 0
\end{equation}
for all functions $f$ on $\Sigma$ such that  
\[
\int f d\Sigma=0.
\]
Here $h$ denote the second fundamental form of the surface.
\end{definition}
\begin{theorem} \label{minimize_self_2} 
Suppose the reference spacetime $\mathfrak{S}$ satisfies the null convergence condition \eqref{null_conv}, and the static slice is spherically symmetric (with a spherically symmetric static potential). Let $X(s)=(\tau(s),X^i(s)), s\in (-\epsilon, \epsilon)$ be a family of isometric embeddings of the same metric $\sigma$ into the interior $\mathring{\mathfrak{S}}$  such that $\Sigma_0=X(0)$ is a sphere of symmetry in the static slice. Then
\[
\begin{split}
\frac{d^2}{ds^2}|_{s=0} E(\Sigma_0,X(s), \frac{\partial}{\partial t}) \ge & 0\\
\end{split}
\]
if $\Sigma_0$ is a stable CMC surface and $\nu(V) \ge 0$.
\end{theorem}
\begin{proof}
Let $H_0(X(s))$ and $\alpha_{H_0}(X(s))$ be the mean curvature vector and the connection one-form in mean curvature gauge of the image of $X(s)$. For simplicity, set $\delta |H_0| = \frac{d}{ds}|_{s=0} |H_0(X(s))|$ and
$\delta \alpha_{H_0} =\frac{d}{ds}|_{s=0} \alpha_{H_0}(X(s))$. Let $\widehat X(s)=(0,X^i(s))$ be the projection of $X(s)(\Sigma)$ onto the static slice. $\widehat X(s)$ is an isometric embedding of the metric 
\[  \hat \sigma(s)_{ab}= \sigma_{ab} +  V ^2(s) \tau_a(s) \tau_b(s) \]
into the static slice and $ \delta \hat \sigma = \frac{d}{ds}|_{s=0} \hat \sigma(s)=0$, as $\tau(0)=0$. Finally, let $\widehat H(s)$ be the mean curvature of $\widehat X(s)$ in the static slice.

We apply Theorem  \ref{thm_first_variation_graph}  to each of $X(s)(\Sigma)$ and conclude that 
\[ 
 \frac{d^2}{ds^2}|_{s=0} E(\Sigma_0,X(s), \frac{\partial}{\partial t}) =  - \frac{1}{8 \pi} \int  ( \delta \alpha_{H_0} )( V  ^2 \nabla f) d\Sigma_0 +\frac{1}{8 \pi} \int \delta V \delta |H_0|  d\Sigma_0. 
\]
The first integral is non-negative as in the proof of the Theorem \ref{minimize_self_1}. For the second integral, we observe that 
\[ \delta |H_0| = \delta \widehat H. \]
We decompose $\delta \widehat X$ into tangential and normal parts to $\Sigma_0$. Let
\[ \delta \widehat X = P^a \frac{\partial \widehat X}{\partial v^a} + \beta \nu.  \]
The components $\beta$ and $P^a$ satisfy
\begin{equation}\label{gauge_first_variation_metric}
2 \beta h_{ab} + \nabla_a  P_b +  \nabla_b  P_a=0
\end{equation}
since  $\delta \hat \sigma=0$.

Taking the trace  of \eqref{gauge_first_variation_metric} and integrating, we conclude that 
\[  \int \beta H_0 d\Sigma_0 = 0. \] In particular, $\int \beta d\Sigma = 0$ since $H_0$ is a constant. 
Since $V$, $H_0$ and $\nu(V)$ are constants on $\Sigma_0$, integrating over $\Sigma_0$ gives
\[
 \int \delta V \delta \widehat H  d\Sigma_0 =  \nu(V) \int   \beta (-\Delta \beta - (|h|^2+Ric(\nu,\nu))\beta) d \Sigma_0.
\]
Integrating by parts, we see that the right hand side is non-negative if $\Sigma_0$ is a stable CMC surface and $\nu(V) $ is non-negative.
\end{proof}
\begin{remark}
For a static spacetime with metric
\[ \check{g}=  -V^2(r)dt^2+\frac{1}{V^2(r)}dr^2+r^2dS^2, \]
the null convergence condition and the stable CMC condition can be expressed explicitly in terms of $V(r)$ and its derivatives. See \cite{Brendle,WWZ}.
\end{remark}

\end{document}